\documentclass{amsart}

\usepackage{amssymb}
\usepackage{amsthm}
\usepackage{amsmath}
\usepackage{enumerate}
\usepackage{hyperref}
\usepackage{mathrsfs}

\theoremstyle{plain}

\newtheorem{thm}{Theorem}[section]
\newtheorem{cor}[thm]{Corollary}
\newtheorem{lem}[thm]{Lemma}

\newcommand{\FF}[1]{\mathbb F_{#1}}

\newcommand{\sbs}{\subseteq}

\renewcommand{\l}{\langle}
\renewcommand{\r}{\rangle}

\newcommand{\nor}{\vartriangleleft}

\newcommand{\End}{\operatorname{End}}

\newcommand{\sym}{\operatorname{Sym}}

\renewcommand{\t}{\times}

\newcommand{\vbasis}[1]{\{v_1,\ldots,v_{#1}\}}

\begin{document}
\title{The minimal base size for a $p$-solvable linear group}

\author{Zolt\'an Halasi} \address{Department of Algebra and Number
  Theory, Institute of Mathematics, University of Debrecen, 4010,
  Debrecen, Pf. 12, Hungary} \email{halasi.zoltan@renyi.mta.hu}

\author{Attila Mar\'oti} \address{Fachbereich
Mathematik, Technische Universit\"{a}t Kaiserslautern, Postfach
3049, 67653 Kaiserslautern, Germany \and Alfr\'ed R\'enyi Institute of
  Mathematics, Re\'altanoda utca 13-15, H-1053, Budapest, Hungary}
\email{maroti@mathematik.uni-kl.de \and maroti.attila@renyi.mta.hu}

\thanks{The research of the first author leading to these results has
  received funding from the European Union's Seventh Framework
  Programme (FP7/2007-2013) under grant agreement no. 318202, from ERC
  Limits of discrete structures Grant No.\ 617747 and from OTKA
  K84233.  The research of the second author was supported by a Marie
  Curie International Reintegration Grant within the 7th European
  Community Framework Programme, by an Alexander von Humboldt
  Fellowship for Experienced Researchers, by the J\'anos Bolyai
  Research Scholarship of the Hungarian Academy of Sciences, by OTKA
  K84233, and by the MTA RAMKI Lend\"ulet Cryptography Research
  Group.}


\begin{abstract}
  Let $V$ be a finite vector space over a finite field of order $q$
  and of characteristic $p$. Let $G\leq GL(V)$ be a $p$-solvable
  completely reducible linear group. Then there exists a base for $G$
  on $V$ of size at most $2$ unless $q \leq 4$ in which case there
  exists a base of size at most $3$. The first statement extends a
  recent result of Halasi and Podoski and the second statement
  generalizes a theorem of Seress. An extension of a theorem of
  P\'alfy and Wolf is also given.
\end{abstract}
\maketitle
\begin{center}
{\it Dedicated to the memory of \'Akos Seress.}
\end{center}

\bigskip
\section{Introduction}
For a finite permutation group $H \leq \mathrm{Sym}(\Omega)$, a
subset of the finite set $\Omega$ is called a base, if its
pointwise stabilizer in $H$ is the identity. The minimal base size
of $H$ (on $\Omega$) is denoted by $b(H)$. Notice that $|H| \leq
{|\Omega|}^{b(H)}$.

One of the highlights of the vast literature on base sizes of
permutation groups is the celebrated paper of \'A.~Seress
\cite{seress} in which it is proved that $b(H) \leq 4$ whenever $H$ is
a solvable primitive permutation group. Since a solvable primitive
permutation group is of affine type, this result is equivalent to
saying that a solvable irreducible linear subgroup $G$ of $GL(V)$ has
a base of size at most $3$ (in its natural action on $V$) where $V$ is
a finite vector space.

There are a number of results on base sizes of linear groups. For
example, D.~Gluck and K.~Magaard \cite[Corollary 3.3]{GM} have shown
that a subgroup $G$ of $GL(V)$ with $(|G|,|V|) = 1$ admits a base of
size at most $94$. If in addition it is assumed that $G$ is
supersolvable or of odd order then $b(G) \leq 2$ by results of
T.R.~Wolf \cite[Theorem A]{W} and S.~Dolfi \cite[Theorem
1.3]{Dolfi_2}. Later S.~Dolfi \cite[Theorem 1.1]{dolfi_biz} and
E.P.~Vdovin \cite[Theorem 1.1]{vdovin_biz} generalized this result to
solvable coprime linear groups. Finally, Z.~Halasi and K.~Podoski
\cite[Theorem 1.1]{HP} improved this result significantly, by proving that
even the solvability assumption can be dropped, and $b(G) \leq 2$ for
any coprime linear group $G$.

We note that for a solvable subgroup $G$ of $GL(V)$ acting
completely reducibly on $V$ we have $b(G) \leq 2$ if the Sylow
$2$-subgroups of $GV$ are Abelian (see \cite[Theorem
2]{DolfiJabara}) or if $|G|$ is not divisible by $3$ (see
\cite[Theorem 2.3]{Yang}).

The following definition has been introduced by M. W. Liebeck and
A. Shalev in \cite{LS}. For a linear group $G\leq GL(V)$ we say
that $\vbasis{k}\sbs V$ is a strong base for $G$ if any element of
$G$ fixing $\langle v_i\rangle$ for every $1\leq i\leq k$ is a
scalar transformation. The minimal size of a strong base for $G$
is denoted by $b^*(G)$. It is known that $b(G) \leq b^{*}(G) \leq
b(G)+1$ (see \cite[Lemma 3.1]{LS}). Furthermore, also $b^*(G)\leq
2$ holds for coprime linear groups by \cite[Lemma 3.3 and Theorem
1.1]{HP}.

The following theorem extends the
above-mentioned result of Seress \cite{seress} and that of Halasi and
Podoski to $p$-solvable groups.
\begin{thm}\label{thm:HP-altalanositas}
  Let $V$ be a finite vector space over a field of order $q$ and of
  characteristic $p$. If $G \leq GL(V)$ is a $p$-solvable group acting
  completely reducibly on $V$, then $b^{*}(G) \leq 2$ unless $q \leq
  4$. Moreover if $q \leq 4$ then $b^{*}(G) \leq 3$.
\end{thm}

One of the motivations of Seress \cite{seress} was a famous result of
P.P.~P\'alfy \cite[Theorem 1]{Palfy} and Wolf \cite[Theorem 3.1]{Wolf}
stating that a solvable primitive permutation group of degree $n$ has
order at most $24^{-1/3}n^{d}$ where $d = 1+\log_9 (48\cdot
24^{1/3})=3.243\ldots$, that is to say, a solvable irreducible subgroup
$G$ of $GL(V)$ has size at most $24^{-1/3}{|V|}^{d-1}$. (This bound is
attained for infinitely many groups.) In the following we extend this
result to $p$-solvable linear groups $G$.
\begin{thm}\label{Palfy-Wolf}
  Let $V$ be a finite vector space over a field of characteristic
  $p$. If $G \leq GL(V)$ is a $p$-solvable group acting completely
  reducibly on $V$, then $|G| \leq 24^{-1/3}|V|^{d-1}$ where $d$ is as
  above.
\end{thm}

We note that the bounds in Theorem \ref{thm:HP-altalanositas} are
best possible for all values of $q$. Indeed, there are infinitely
many irreducible solvable linear groups $G \leq GL(V)$ with
$|G| > {|V|}^{2}$ for $q=2$ or $3$ (see \cite[Theorem 1]{Palfy} or
\cite[Proposition 3.2]{Wolf}) and there are even infinitely many odd
order completely reducible linear groups $G \leq GL(V)$ with $|G|
> |V|$ for $q \geq 5$ (see \cite[Theorem 3B]{Palfy2} and the
remark that follows). For $q=4$ we note that there are primitive,
irreducible solvable linear subgroups $H$ of $GL(3,4)$ with
$b(H)=3$ and thus there are infinitely many imprimitive,
irreducible solvable linear groups $G = H \wr S \leq GL(3r,4)$
with $b(G)=3$ where $S$ is a solvable transitive permutation group
of degree $r$.

Theorem \ref{thm:HP-altalanositas} has been applied in \cite{CHMN} to
Gluck's conjecture.

\section{Preliminaries}
Throughout this paper let $\FF q$ be a finite field of characteristic
$p$ and let $V$ be an $n$-dimensional vector space over $\FF
q$. Furthermore, let $G\leq GL(V)$ be a linear group acting on $V$ in
the natural way, let $b(G)$ denote its minimal base size, and let
$b^{*}(G)$ denote its minimal strong base size (both notions defined
in Section 1).

If the vector space $V$ is fixed, then the group of scalar
transformations of $V$ (the center of $GL(V)$) will be denoted by
$Z$. Thus $Z\simeq \FF q^\t$, the multiplicative group of the base
field. As $G\leq GL(V)$ is $p$-solvable if and only
if $G Z$ is $p$-solvable, we can (and we will)
always assume, in the proofs of Theorems \ref{thm:HP-altalanositas} and
\ref{Palfy-Wolf}, that $G$ contains $Z$. After choosing a basis
$\vbasis{n}\sbs V$, we will always identify the group $GL(V)$ with the
group $GL(n,q)$.

Put $t(q) =3$ for $q \leq 4$ and $t(q) = 2$ for $q \geq 5$.

Finally, if $G\leq GL(V)$ and $X\sbs V$, then $C_G(X)=\{g \in G \,
\mid \, g(x)=x \, \, \, \forall\,x\in X\}$ and $N_G(X)=\{ g \in G
\, \mid g(x)\in X \, \, \forall \, x\in X\}$ will denote the
pointwise and setwise stabilizer of $X$ in $G$, respectively.
\section{Special bases in linear groups}
In this section we will show that there exist bases of special kinds
for certain linear groups. As a consequence (Corollary
\ref{cor:strong_basis}), we derive that it is sufficient to establish
the required bounds in Theorem \ref{thm:HP-altalanositas} for $b(G)$
rather than for $b^{*}(G)$.
\begin{thm}\label{thm:jo_basis}
  Let $V$ be an $n$-dimensional vector space over $\FF q$, a field of
  characteristic $p$ and let $Z \leq G \leq GL(V)$ be a $p$-solvable
  linear group.
  \begin{enumerate}
  \item If $n=2$ and $q \geq 5$, then
    at least one of the following holds.
    \begin{enumerate}
    \item There is a basis $x,y\in V$ such that
      $N_G(\l x\r)\sbs N_G(\l y\r)$.
    \item $p=2$ and there is a basis $x, y \in V$ such that $N_{G}(\l
      x \r) = Z \times C_2$ and the involution $g$ in $N_{G}(\l x
      \r)$ satisfies $g(x) = x$ and $g(y) = y + x$.
    \end{enumerate}
  \item If $n = 3$ and $q = 3$ or $4$, then at
    least one of the following holds.
    \begin{enumerate}
    \item There is a basis $x,y,z\in V$ such that
      $N_G(\l x\r)\cap N_G(\l y\r)\sbs N_G(\l z\r)$.
    \item There is a basis $x,y,z\in V$ such that
      $N_G(\l y,z\r)=G$.
    \end{enumerate}
  \end{enumerate}
\end{thm}
\begin{proof}
  Firstly we may assume that $G$ is an irreducible primitive subgroup
  of $GL(V)$. Since $G$ is $p$-solvable by assumption, we see that $G$
  does not contain $SL(V)$.

  First consider statement (1). By considering the action of $G$ on
  the set $S$ of $1$-dimensional subspaces of $V$, we may assume that
  the number of Sylow $p$-subgroups of $G$ is equal to $|S| =
  q+1$. For otherwise there exists $\langle x \rangle \in S$ whose
  stabilizer in $G$ is a $p'$-group and thus Maschke's theorem gives
  1/(a). For $q=p$ any subgroup of $GL(V)$ with $q+1$ Sylow
  $p$-subgroups contains $SL(V)$, so in this case we are done. So
  assume that $q > p$.

  Since $G$ acts transitively on the set of Sylow $p$-subgroups of $G$
  and every Sylow $p$-subgroup stabilizes a unique subspace in $S$, it
  follows that $G$ acts transitively on $S$. Moreover since $Z \leq G$
  it also follows that $G$ acts transitively on the set of non-zero
  vectors of $V$.

  By Hering's theorem (see \cite[Chapter XII, Remark 7.5\ (a)]{HB}) we
  see that if $q$ is odd (and not a prime by assumption) then $q$ must
  be $9$ and $G$ has a normal subgroup isomorphic to $SL(2,5)$ (case
  (5)). But then $G$ is not $3$-solvable and so we can rule out this
  possibility.  Similarly, if $q$ is even, then the only possibility
  is that $G\geq Z$ normalizes a Singer cycle $GL(1,q^2)$ (case
  (1)). The only such group not satisfying 1/(a) is the full
  semilinear group $\Gamma(1,q^2)\simeq GL(1,q^2).2$. In this case
  taking $x$ to be any non-zero vector in $V$ we have $N_{G}(\l x \r)
  = Z \times C_2$ and the involution $g$ in $N_{G}(\l x \r)$ satisfies
  $g(x) = x$ and $g(y) = y+x$ for some $y\in V$.

  Finally, statement (2) has been checked with GAP \cite{GAP} by using
  the list of all primitive permutation groups of degrees 27 and 64,
  respectively.
\end{proof}
As a direct consequence we get the following.
\begin{cor}\label{cor:jo_basis}
  Let us assume that $Z \leq G\leq GL(V)$ is a $p$-solvable linear group with
  $b(G)\leq t(q)$.
  \begin{enumerate}
  \item If $q \geq 5$, then one of the following holds.
    \begin{enumerate}
    \item There exists a base $x,y\in V$ such that $N_G(\l x\r)\cap
      N_G(\l x,y\r)\sbs N_G(\l y\r)$.
    \item $p=2$ and there exists a base $x,y\in V$ such that any
      non-identity element of $C_G(x)\cap N_G(\l x,y\r)$ takes $y$ to
      $y+x$.
    \end{enumerate}
  \item If $q \leq 4$, then at least one of the following holds.
    \begin{enumerate}
    \item There exists a base $x,y,z\in V$ such that
      \[N_G(\l x\r)\cap N_G(\l y\r)\cap N_G(\l x,y,z\r)\sbs N_G(\l z\r).\]
    \item There exists a base $x,y,z\in V$ such that
      $N_G(\l x,y,z\r) \sbs N_G(\l y,z\r)$ with $x\notin \l y,z\r$.
    \end{enumerate}
  \end{enumerate}
\end{cor}
\begin{proof}
  First, 1/(a) or 2/(a) holds if $\dim (V) < t(q)$ so assume that
  $\dim(V) \geq t(q)$. Both parts of the corollary can be proved
  by choosing a subspace $U\leq V$ of dimension $t(q)$ generated by a
  base for $G$ and by restricting $N_G(U)$ to this subspace. Notice
  that the image of this restriction is also $p$-solvable, so
  Theorem \ref{thm:jo_basis} can be applied.
\end{proof}
\begin{cor}\label{cor:strong_basis}
  Let $V$ be a vector space over the field $\FF q$ of characteristic
  $p$. Let $Z \leq G \leq GL(V)$ be $p$-solvable with
  $b(G)\leq t(q)$. Then $b^*(G)\leq t(q)$.
\end{cor}
\begin{proof}
  We may assume that $\dim (V) \geq t(q)$ and that $q > 2$.
  Let us choose a base for $G$ of size $t(q)$ satisfying the property
  given in Corollary \ref{cor:jo_basis}. For $q \geq 5$, if $x,y\in V$ is
  such a base, then $x,x+y$ is a strong base for $G$. Likewise, for
  $q = 3$ or $4$, if $x,y,z\in V$ is a base satisfying (2/a) of
  Corollary \ref{cor:jo_basis}, then $x,y,x+y+z$ is a strong base for
  $G$. Finally, in case $x,y,z\in V$ is a base for $G$ satisfying
  (2/b) of Corollary \ref{cor:jo_basis}, then
  $x,y+x,z+x$ is a strong base for $G$.
\end{proof}
\section{Further reductions}\label{sec:notirred}
Let us use induction on the dimension $n$ of $V$ in the proofs of
Theorems \ref{thm:HP-altalanositas} and \ref{Palfy-Wolf}. The case $n =
1$ is clear. Let us assume that $n > 1$ and that both Theorems
\ref{thm:HP-altalanositas} and \ref{Palfy-Wolf} are true for dimensions
less than $n$.

First we reduce the proof of both theorems for the case when $G \leq GL(V)$ acts
irreducibly on $V$. For otherwise let $V=V_1\oplus
V_2\oplus\ldots\oplus V_k$ be a decomposition of $V$ to irreducible
$\FF q G$-modules. 

By induction, there exist vectors $x_{i,1}, \ldots ,
x_{i,t(q)}$ in $V_{i}$ for $1 \leq i \leq k$ with the property that
$C_{G}(\{ x_{i,1}, \ldots , x_{i,t(q)} \})$ is precisely the kernel of
the action of $G$ on $V_{i}$. Now put $x_{j} = \sum_{i=1}^{k}
x_{i,j}$ for $1 \leq j \leq t(q)$. One can see that $C_{G}(\{ x_{1},
\ldots , x_{t(q)} \}) = \cap_{i=1}^k C_G(V_i) = 1$.

For Theorem \ref{Palfy-Wolf} notice that $G$ is a subgroup of a
direct product $\times_{i=1}^k H_{i}$ of $p$-solvable groups $H_i$ acting
irreducibly and faithfully on the $V_{i}$'s. Hence we have 
\[
|G|\leq \prod_{i=1}^k |H_i|\leq \prod_{i=1}^k \left(24^{-1/3} |V_i|^{d-1}\right)
= 24^{-k/3}|V|^{d-1}
\] 
by induction.

So from now on we will assume that $G \leq GL(V)$ acts
irreducibly on $V$.

For Theorem \ref{thm:HP-altalanositas} we may also assume that $q
\not= 2$, $4$. Otherwise, $G$ is solvable by the Odd Order Theorem
and we can use the result of Seress \cite{seress}.

For Theorem \ref{Palfy-Wolf} we may assume that $|G|>|V|^2$. If
$|G| \leq |V|^{2}$ then ${|V|}^{2} < 24^{-1/3}{|V|}^{d-1}$ for
$|V| \geq 79$, so we may assume that $|V| \leq 73$. If $|V|$ is a
prime or $p=2$ then $G$ is solvable and the theorem of P\'alfy
\cite{Palfy} and Wolf \cite{Wolf} can be applied. Hence the cases
$|V| = 5^{2}, 7^{2}, 3^2$ or $3^3$ remain to be examined. But in
these cases there is no non-solvable, $p$-solvable irreducible
subgroup of $GL(V)$ (see \cite{GAP}).

Now, if $b(G)\leq 2$ then $|G|\leq |V|^2$. So, once Theorem 1.1 is
proved, it remains to prove Theorem \ref{Palfy-Wolf} only in case
$q=3$ and $b(G)>2$.
\section{Imprimitive linear groups}
In this section we show that we may assume (for the proofs of
Theorems \ref{thm:HP-altalanositas} and \ref{Palfy-Wolf}) that $G$
is a primitive (irreducible) subgroup of $GL(V)$.

We first consider Theorem \ref{thm:HP-altalanositas}.

For $G\leq GL(V)$ an irreducible imprimitive linear group, let
$V=V_1\oplus\cdots\oplus V_k$ be a decomposition of $V$ into
subspaces such that $G$ permutes these subspaces in a transitive
and primitive way. This action of $G$ defines a homomorphism from
$G$ into the symmetric group $\sym(\Omega)$ for
$\Omega=\{V_1,\ldots,V_k\}$ with kernel $N$.

The factor group $G/N\leq S_k$ is $p$-solvable, so it does not involve $A_q$
for $q\geq 5$ and it does not involve $A_5$ for $q=3$.  By using
\cite[Theorem 2.3]{HP} it follows that for $q\geq 5$ there
is a vector $a=(a_1,\ldots,a_k)\in \FF q^k$ such that
$C_{G/N}(a)=1$, while for $q=3$ there is a pair of vectors
$a=(a_1,\ldots,a_k),\ b=(b_1,\ldots,b_k)\in \FF 3^k$ such that
$C_{G/N}(a)\cap C_{G/N}(b)=1$. (Here, $G/N$ acts on $\FF q^k$ by
permuting coordinates.)

In fact for $q \geq 8$ even we can say a bit more. For such a $q$ let
$S$ be a subset of $\FF q$ of size $q/2$ with the property that
for each $c \in \FF q$ exactly one of $c$ and $c+1$ is contained
in $S$. By \cite[Lemma 1/(c)]{Dolfi} there exists a vector $a =
(a_{1}, \ldots , a_{k}) \in S^{k}$ such that $C_{G/N}(a)=1$.

For each $1\leq i\leq k$ let $H_i=N_G(V_i)$, so $N=\cap_i H_i$. By
induction (on the dimension), there is a base in $V_1$ of size
$t(q)$ for $H_1/C_{H_1}(V_1)$.

Now we can use Corollary \ref{cor:jo_basis}. First let $q\geq 5$.
Then there is a base $x_1,y_1\in V_1$ for
$K_1=H_1/C_{H_1}(V_1)\leq GL(V_1)$ such that $N_{K_1}(\l x_1\r)
\cap N_{K_{1}}(\l x_{1},y_{1} \r) \sbs N_{K_1}(\l y_1\r)$ or that
any non-identity element of $C_{K_1}(x_1) \cap N_{K_{1}}(\l x_{1},y_{1}
\r)$ takes $y_{1}$ to $y_{1} + x_{1}$.

Let $\{g_1=1,g_2,\ldots,g_k\}$ be a set of left coset
representatives for $H_1$ in $G$ and $x_i=g_ix_1,\ y_i=g_iy_1$ for
every $i$. Now let
\[
x=\sum_{i=1}^k x_i,\qquad  y=\sum_{i=1}^k y_i+a_ix_i.
\]

In case $q=3$ let $x_1,y_1,z_1\in V_1$ be a base for
$K_1=H_1/C_{H_1}(V_1)\leq GL(V_1)$ satisfying (2/a) or (2/b) of
Corollary \ref{cor:jo_basis}. Again, let
$\{g_1=1,g_2,\ldots,g_k\}$ be a set of left coset representatives
for $H_1$ in $G$ and $x_i=g_ix_1,\ y_i=g_iy_1,\ z_i=g_iz_1$ for
every $i$. Depending on which part of part (2) of Corollary
\ref{cor:jo_basis} is satisfied for $x_1,y_1,z_1$ let
\begin{align*}
  x&=\sum_{i=1}^k x_i,&\quad y&=\sum_{i=1}^k y_i&\quad &z=\sum_{i=1}^k
  (z_i+b_ix_i+a_iy_i)&\textrm{ if (2/a) holds},\\
  x&=\sum_{i=1}^k x_i,&\quad y&=\sum_{i=1}^k (y_i+a_ix_i)&\quad &z=\sum_{i=1}^k
  (z_i+b_ix_i)&\textrm{ if (2/b) holds}.
\end{align*}
In each case, it is easy to see that the given set of vectors is a
base for $G$ by using similar arguments as in the proof of 
\cite[Theorem 2.6]{HP}.

Now we turn to the reduction of Theorem \ref{Palfy-Wolf} to primitive
groups. Notice that $N$ is a $p$-solvable group
and $V$ is the sum of at least $k$ irreducible $\FF q N$-modules, so
we have $|N| \leq 24^{-k/3}{|V|}^{d-1}$ by Section
\ref{sec:notirred}. Since the permutation group $G/N \leq S_{k}$ is
$3$-solvable, it does not contain any non-Abelian alternating
composition factor, and so $|G/N| \leq 24^{(k-1)/3}$, by
\cite[Corollary 1.5]{Maroti}. But then $|G| = |N| |G/N| \leq
24^{-1/3}{|V|}^{d-1}$ which is exactly what we wanted.
\section{Groups of semilinear transformations}
In this section we reduce Theorems
\ref{thm:HP-altalanositas} and \ref{Palfy-Wolf} to the case when
every irreducible $\FF qN$-submodule of $V$ is absolutely
irreducible for any normal subgroup $N$ of $G$.

For this purpose let $N\nor G$ be a normal subgroup of $G$. Then
$V$ is a homogeneous $\FF qN$-module, so $V=V_1\oplus
V_2\oplus\cdots\oplus V_k$, where the $V_i$'s are isomorphic
irreducible $\FF qN$-modules. Let $T := \End_{\FF qN}(V_1)$.
Assuming that the $V_i$'s are not absolutely irreducible, $T$ is a
proper field extension of $\FF q$, and
  \[C_{GL(V)}(N)=\End_{\FF qN}(V)\cap GL(V)\simeq GL(k,T).\]
Furthermore, $L=Z(C_{GL(V)}(N))\simeq Z(GL(k,T))\simeq T^\t$. Now,
by using $L$, we can extend $V$ to a $T$-vector space of dimension
$l:=\dim_T V <\dim_{\FF q} V$. As $G\leq N_{GL(V)}(L)$, in this
way we get an inclusion $G\leq \Gamma L(l,T)$. We proceed by
proving the following theorem.
\begin{thm}\label{lem:semi}
  For a proper field extension $T$ of $\FF q$ let $G\leq \Gamma
  L(l,T)$ be a semilinear group acting on the $\FF q$-space $V$ and
  let $H=G\cap GL(l,T)$. Suppose that $G$ is $p$-solvable and that
  $b(H) \leq t(|T|)$. Then $b(G) \leq t(|T|)$.
\end{thm}
\begin{proof}
  We modify the proof of \cite[Lemma 6.1]{HP} to make it work in this
  more general setting.

  Clearly we may assume that $|T|\geq 8$ is different from a prime.
  In these cases $t(|T|)=2$.

  Let $u_{1}, u_{2}$ be a base for $H$. By Corollary
  \ref{cor:jo_basis}, we may also assume that
  $$N_{H}(\langle u_{1} \rangle) \cap N_{H}(\langle u_{1}, u_{2}
  \rangle) \subseteq N_{H}(\langle u_{2} \rangle)$$ or that every
  non-identity element of $C_{H}(u_{1}) \cap N_{H}(\langle u_{1},
  u_{2} \rangle)$ takes $u_{2}$ to  $u_{2} + u_{1}$.
  (The latter case occurs only if $p=2$.)

  For every $\alpha\in T$ let $H_\alpha=C_G(u_1)\cap C_G(u_2+\alpha
  u_1)\leq G$. Our goal is
  to prove that $H_\alpha=1$ for some $\alpha\in T$. If $g\in
  \l\cup H_\alpha\r$, then $g(u_{1}) = u_{1}$ and $g(u_2)=u_2+\delta u_1$ for
  some $\delta\in T$.

  We claim that $|\l\cup H_\alpha\r\cap H| \leq 2$. Let $h \in \l\cup
  H_\alpha\r\cap H$. On the one hand, the action of $h$ on $V$ is
  $T$-linear, since $h \in H$. On the other hand, $h(u_1)=u_1$ and
  $h(u_2)=u_2+\delta u_1$ for some $\delta\in T$. By our assumption
  above, either $h \in N_{H}(\langle u_{2} \rangle)$ and $\delta = 0$,
  or $h$ is an involution and $\delta = 1$. Thus we obtain the claim
  since $C_H(u_1)\cap C_H(u_2) = 1$.

  Let $z$ be the generator of the group $\langle\cup
  H_\alpha\rangle\cap H$. This is a central element in $\langle\cup
  H_\alpha\rangle$. For every $g \in G$ let $\sigma_{g} \in
  \mathrm{Gal}(T | \FF q)$ denote the action of $g$ on $T$.

  Let $g$ and $h$ be two elements of $\langle\cup H_\alpha\rangle$.
  Since $G/H$ is embedded into $\mathrm{Gal}(T | \FF q)$, we get
  $\sigma_{g} \not= \sigma_{h}$ unless $g = h$ or $g = hz$.
  Furthermore, a routine calculation shows that the subfields of $T$
  fixed by $\sigma_{g}$ and $\sigma_{h}$ are the same if and only if
  $\l g \r = \l h \r$ or $\l g \r = \l hz \r$.

  If $g \in H_{\alpha} \cap H_{\beta}$, then $g(u_{2}) = u_{2} +
  (\alpha - \alpha^{\sigma_{g}})u_{1} = u_{2} + (\beta -
  \beta^{\sigma_{g}})u_{1}$, so $\alpha - \beta$ is fixed by
  $\sigma_{g}$. Let $K_{g} = \{ \alpha \in T \mid g \in H_{\alpha}
  \}$. The previous calculation shows that $K_{g}$ is an additive
  coset of the subfield fixed by $\sigma_{g}$, so $|K_{g}| = p^{d}$
  for some $d \mid f = \log_{q}|T|$. Since for any $d \mid f$ there is
  a unique $p^{d}$-element subfield of $T$, we get $|K_{g}| \not=
  |K_{h}|$ unless the subfields fixed by $\sigma_{g}$ and $\sigma_{h}$
  are the same. As we have seen, this means that $\l g \r = \l h \r$
  or $\l g \r = \l hz \r$. Consequently, $|K_{g}| \not= |K_{h}|$
  unless $K_{g} = K_{h}$ or $K_{g} = K_{hz}$. Hence we get
  \[
  |\bigcup_{g \in \cup H_{\alpha} \setminus \{ 1\}} K_{g}| \leq 2
  \sum_{d \mid f, d < f} q^{d} \leq 2 \sum_{d < f} q^{d} < q^{f} = |T|.
  \]
  So there is a $\gamma \in T$ which is not contained in $K_{g}$ for
  any $g \in \cup H_{\alpha} \setminus \{ 1 \}$. This exactly means
  that $H_{\gamma} = C_{G}(u_{1}) \cap C_{G}(u_{2} + \gamma u_{1}) =
  1$.
\end{proof}
Using Theorem \ref{lem:semi}, we can assume that $G\leq GL(l,T)$.
As $l=\dim_T V<\dim_{\FF q}(V)$, we can use induction on the
dimension of $V$, thus $b(G) \leq 2$.

By the last paragraph of Section \ref{sec:notirred}, we need
not consider Theorem \ref{Palfy-Wolf} here.

Hence in the following we assume that $V$ is a direct sum of
isomorphic absolutely irreducible $\FF qN$-modules for any $N\nor
G$.
\section{Stabilizers of tensor product decompositions}
Let $N\nor G$ and let $V=V_1\oplus\cdots\oplus V_k$ be a direct
decomposition of $V$ into isomorphic absolutely irreducible $\FF
qN$-modules. By choosing a suitable basis in $V_1,V_2,\ldots, V_k$, we
can assume that $G\leq GL(n,q)$ such that any element of $N$ is of
the form $A\otimes I_k$ for some $A\in N_{V_1}\leq GL(n/k,q)$. By
using \cite[Lemma 4.4.3(ii)]{KL} we get
\[
N_{GL(n,q)}(N)=\{B\otimes C\,|\,B\in N_{GL(n/k,q)}(N_{V_1}),\ C\in GL(k,q)\}.
\]
Let
\[
G_1=\{g_1\in GL(n/k,q)\,|\,\exists g\in G,g_2\in GL(k,q)
\textrm{ such that }g=g_1\otimes g_2\}.
\]
We define $G_2\leq GL(k,q)$ in an analogous way. Then $G\leq
G_1\otimes G_2$. Here $G/Z\simeq (G_1/Z)\times
(G_2/Z)$, hence $G_1\leq GL(n/k,q)$ and $G_2\leq GL(k,q)$ are 
$p$-solvable irreducible linear groups. 
If $1<k<n$, then by using induction
for $G_1\leq GL(n/k,q)$ and $G_2\leq GL(k,q)$ we get
$b(G_1)\leq t(q)$ and $b(G_2)\leq t(q)$. Furthermore $b^{*}(G_1) \leq
t(q)$ and $b^{*}(G_2) \leq t(q)$ by Corollary
\ref{cor:strong_basis}. Thus \cite[Lemma 3.3 (ii)]{LS} gives us
\begin{align*}
  b(G)&\leq b(G_1\otimes G_2) \leq b^*(G_1\otimes G_2)\leq \\
  &\max(b^*(G_1),b^*(G_2)) \leq t(q).
\end{align*}
For the reduction of Theorem \ref{Palfy-Wolf}, by using induction on the
dimension, we have
\[
|G| \leq |G_{1}| \cdot |G_{2}| \leq 24^{-1/3} {q}^{(n/k)(d-1)} \cdot
24^{-1/3}{q}^{k(d-1)} \leq 24^{-1/3}{|V|}^{d-1}.
\]
Thus, from now on we can assume that for every normal
subgroup $N\nor G$ either $N\leq Z$ or $V$ is absolutely
irreducible as an $\FF qN$-module.
\section{Groups of symplectic type}
\label{symplectic}

From now on assume that $N$ is a normal subgroup of $G$ containing $Z$
such that $N/Z$ is a minimal normal subgroup of $G/Z$.  Then $N/Z$
is a direct product of isomorphic simple groups. In this section we
examine the situation when $N/Z$ is an elementary Abelian group.

If $N$ is Abelian then it is central in $G$. So assume that $N$ is
non-Abelian.

If $N/Z$ is elementary Abelian of rank at least $2$, then $G$ is of
symplectic type. Such groups were examined in \cite[Section 5]{HP}
(see also \cite[Remark 5.20]{HP}) where it was proved that $b(G)\leq
2$ unless $q\in \{3,4\}$, when $b(G)\leq 3$ holds.

For the reduction of Theorem \ref{Palfy-Wolf}, we need only
examine the case $q=3,\ n=2^k$. For this we can use the fact that
$G/N$ can be considered as a subgroup of the symplectic group
$\mathrm{Sp}(2k,2)$. By the theorem of P\'alfy \cite{Palfy} and
Wolf \cite{Wolf}, we may assume that $G$ is a non-solvable (and
$3$-solvable) group. Thus we must have a composition factor of $G$
(and thus of $G/N$) isomorphic to a Suzuki group. Since the
smallest Suzuki group $\mathrm{Suz}(8)$ has order larger than
$|\mathrm{Sp}(4,2)|$, we must have $k \geq 3$. On the other hand,
since the second largest Suzuki group $\mathrm{Suz}(32)$ has order
larger than $|\mathrm{Sp}(6,2)|$ and since $\mathrm{Suz}(8)$ is
not a section of $\mathrm{Sp}(6,2)$ (since $13$ divides the order
of the first group but not the order of the second), we see that
$k \not= 3$. But for $k \geq 4$ we clearly have $|G|=|N| |G/N| <
2^{2k^{2} + 3k + 3} < 24^{-1/3}|V|^{d-1}$, by use of the formula
for the order of $\mathrm{Sp}(2k,2)$.
\section{Tensor product actions}
Now let $N/Z$ be a direct product of $t\geq 2$ isomorphic
non-Abelian simple groups. Then $N=L_1\star L_2\star \cdots \star
L_t$ is a central product of isomorphic groups such that for every
$1\leq i\leq t$ we have $Z\leq L_i,\ L_i/Z$ is simple.
Furthermore, conjugation by elements of $G$ permutes the subgroups
$L_1,L_2,\ldots, L_t$ in a transitive way. By choosing an
irreducible $\FF qL_1$-module $V_1\leq V$, and a set of coset
representatives $g_1=1,g_2,\ldots,g_t\in G$ of $G_1=N_G(V_1)$ such
that $L_i=g_iL_1g_i^{-1}$, we get that $V_i:=g_iV_1$ is an
absolutely irreducible $\FF qL_i$-module for each $1\leq i\leq t$.
Now, $V\simeq V_1\otimes V_2\otimes \cdots \otimes V_t$ and $G$
permutes the factors of this tensor product. It follows that $G$
is embedded into the central wreath product $G_1\wr_c S_t$.
Clearly $G_{1}\leq GL(V_1)$ is a $p$-solvable irreducible linear group. 
Thus $b(G_{1}) \leq t(q)$ and $b^{*}(G_{1}) \leq t(q)$ by induction on
the dimension $m$ of $V_{1}$ and by Corollary
\ref{cor:strong_basis}.

First let $q\geq 5$. Then $t(q)=2$. Thus $b(G)\leq 2$ follows from
\cite[Theorem 3.6]{HP} unless $(m,t)=(2,2)$. In case $(m,t)=(2,2)$,
that is, $G\leq G_1\wr_c S_2\leq GL(4,q)$ for some $p$-solvable group
$G_1\leq GL(2,q)$ let $x_1,y_1\in V_1$ be a basis of $V_1$ satisfying
either $N_{G_1}(\l x_1\r) \sbs N_{G_1}(\l y_1\r)$ or the property that
every non-identity element of $C_{G_1}(x_1)$ takes $y_{1}$ to $y_{1} +
x_{1}$. (Such a basis exists by Theorem \ref{thm:jo_basis}.) Now, it
is easy to see that by choosing any $\alpha\in \FF q\setminus\{0,1\}$ we get
that $x_1\otimes x_1,\ y_1\otimes (y_1+\alpha x_1)$ is a base for
$G_1\wr_c S_2\geq G$.

Now, let $q=3$. Let $x_1,y_1,z_1\in V_1$ be a strong base for
$G_1$. Then the stabilizer of $\underbrace{x_1\otimes
  x_1\otimes\cdots\otimes x_1}_{t\textrm{ factors}}\in V$ is of the
form $H=H_1\wr_c S_t$, where $y_1,z_1\in V_1$ is a strong base for
$H_1=N_{G_1}(x_1)$, so $b^*(H_1)\leq 2$. If $(m,t)\neq (2,2)$ then
$b(H)\leq 2$ by \cite[Theorem 3.6]{HP}, which results in $b(G)\leq
3$. Finally, let $(m,t)=(2,2)$. By choosing a basis $x_1,y_1\in
V_1$, it is easy to see that $x_1\otimes x_1,y_1\otimes
y_1,x_1\otimes y_1\in V$ is a base for $GL(V_1)\wr_c S_2\geq G$.

As for the order of $G$ notice that $G \leq G_1\wr_c S$ where $S
\leq S_t$ is a $3$-solvable group. Thus by induction and by
\cite[Corollary 1.5]{Maroti} we have
\[
|G| \leq |G_{1}|^{t} |S| \leq 24^{-t/3}{|V_{1}|}^{(d-1)t} 24^{(t-1)/3}
= 24^{-1/3}{|V|}^{d-1}.
\]
\section{Almost quasisimple groups}
Finally, let $Z\leq N\nor G$ be such that $N/Z$ is a non-Abelian
simple group. Let $N_1=[N,N]\nor G$ and let $V_1$ be an irreducible $\FF
pN_1$-submodule of $V$ and $G_1=\{g\in G\,|\,g(V_1)=V_1\}$ be the
stabilizer of $V_1$. By using the same argument as 
in the last paragraph of \cite[Page 29]{HP} we get that 
$G_1$ is included in $GL(V_1)$ and we have a chain of subgroups
$N_1\nor G_1\leq GL(V_1)$ where
$G_1$ is $p$-solvable, $N_1$ is quasisimple and $V_1$ is
irreducible as an $\FF p N_1$-module.

Suppose that $b(G_{1}) \leq 2$ in the action of $G_1$ on $V_1$,
that is, there exist $x,y\in V_1\leq V$ such that $C_{G_1}(x)\cap
C_{G_1}(y)=1$. For any element $g\in G$ with $g(x)=x$ we have that
$N_1x=\{nx\,|\,n\in N_1\}$ is a $g$-invariant subset. As the
$\FF p$-subspace generated by $N_1 x$ is exactly $V_1$, we get
that $g\in G_1$. This proves that
$C_G(x)\cap C_G(y)=C_{G_1}(x)\cap C_{G_1}(y)=1$. Thus $b(G)\leq
2$.

Hence if we manage to show that $b(G_{1}) \leq 2$ then we are
finished with the proofs of both Theorems \ref{thm:HP-altalanositas} and
\ref{Palfy-Wolf}.

So assume that $G = G_{1}$ and $V = V_1$. Moreover, by the
previous sections, we have that $q=p$. Also $N = N_{1}$. To
summarize, $G \leq GL(V)$ is a group having a quasisimple
irreducible normal subgroup $N$ containing $Z$.

We claim that $G/Z$ is almost simple. For this it is sufficient to see that $N/Z$ is the unique minimal normal subgroup of $G/Z$. For let $M/Z$ be another minimal normal subgroup of $G/Z$. By Section \ref{symplectic}, we may assume that $M/Z$ is non-Abelian. Furthermore the group $MN$ is a central product and so $[M,N] = 1$. But this is impossible since the centralizer of $N$ in $G$ must be Abelian. 
\begin{lem}
  \label{l1} If $N$ has a regular orbit on $V$ then $b(G) \leq 2$.
\end{lem}
\begin{proof}
  Since $N$ is normal in $G$ a regular $N$-orbit $\Delta$ containing a
  given vector $v$ is a block of imprimitivity inside the $G$-orbit
  containing $v$. Hence the group $C_{G}(v) N$ is transitive on
  $\Delta$ and $N$ is regular on $\Delta$. Thus for every $h \in
  C_{G}(v)$ the number $|\mathrm{fix}(h)|$ of fixed points of $h$ on
  $\Delta$ is $|C_{N}(h)|$. To prove that $G$ has a base of size at
  most $2$ on $V$, it is sufficient to see that there exists a vector
  $w$ in $\Delta$ that is not fixed by any non-trivial element of
  $C_{G}(v)$.

  First notice that if $N/Z(N)$ is isomorphic to the non-Abelian
  finite simple group $S$ then $|C_{G}(v)| \leq |\mathrm{Out}(S)| <
  m(S)$ where $m(S)$ is the minimal index of a proper subgroup of
  $S$. This latter inequality follows from \cite[Lemma 2.7 (i)]{AG}.

  But
  \[
  \sum |\mathrm{fix}(h)| = \sum |C_{N}(h)| < |C_{G}(v)| \cdot \frac{|N|}{m(S)}
  < |N|
  \]
  where the sums are over all non-identity elements $h$ in
  $C_{G}(v)$. This completes the proof of the lemma.
\end{proof}
By Lemma \ref{l1}, in the following we may assume that $N$ does not
have a regular orbit on $V$.  Our final theorem finishes the proofs of
Theorems \ref{thm:HP-altalanositas} and \ref{Palfy-Wolf}.
\begin{thm}
  Under the current assumptions $G$ is a $p'$-group and $b(G) \leq 2$.
\end{thm}
\begin{proof}
By using Goodwin's theorem \cite[Theorem 1]{Goodwin}, K\"ohler and
  Pahlings \cite[Theorem 2.2]{KP} gave a complete list of (irreducible) quasisimple
  $p'$-groups $N$ such that $N$ does not have a regular orbit on
  $V$. In all these exceptional cases, when $N/Z$ is simple, $|\mathrm{Out}(N/Z)|$ is divisible by no prime larger than $3$ while $p$ is always at least $5$. So $G$ itself is a $p'$-group. But then $G$ admits a base of size $2$ on
  $V$ by \cite[Theorem 4.4]{HP}.
\end{proof}

\end{document}